\documentclass{amsart}
\usepackage{amsthm}
\usepackage{amssymb}
\usepackage{colonequals}
\usepackage{dsfont}
\usepackage{tikz-cd}
\usepackage{enumerate}
\usepackage{eucal}
\usepackage{hyperref}
\usepackage[all]{xy}
\hypersetup{%
  bookmarksnumbered=true,%
  colorlinks=true,%
  linkcolor=blue,%
  citecolor=blue,%
  filecolor=blue,%
  menucolor=blue,%
  urlcolor=blue,%
  bookmarksopen=true,%
  bookmarksdepth=2,%
  pageanchor=true}

\usepackage{mathtools}
\usepackage{todonotes}

\makeatletter
\@namedef{subjclassname@2020}{\textup{2020} Mathematics Subject Classification}
\makeatother

\numberwithin{equation}{section}

\theoremstyle{plain}
\newtheorem{theorem}[equation]{Theorem}
\newtheorem{proposition}[equation]{Proposition}
\newtheorem{lemma}[equation]{Lemma} 
\newtheorem{corollary}[equation]{Corollary}

\theoremstyle{definition}

\theoremstyle{remark}
\newtheorem{remark}[equation]{Remark} 

\newtheorem*{ack}{Acknowledgements}

\newcommand{\bbZ}{\mathbb Z}

\newcommand{\bfi}{\mathbf{i}}

\newcommand{\bSpec}{\operatorname{Spc}}

\newcommand{\colim}{\operatorname{colim}}

\newcommand{\cat}{\mathcal}

\newcommand{\dbcat}[1]{{\mathbf D}^{\mathrm{b}}(\operatorname{mod}#1)}
\newcommand{\dcat}[1]{\mathbf{D}(#1)}
\newcommand{\duals}[1]{{#1}^{\operatorname{d}}}

\newcommand{\End}{\operatorname{End}}

\newcommand{\fa}{\mathfrak{a}}
\newcommand{\fm}{\mathfrak{m}}
\newcommand{\fp}{\mathfrak{p}}
\newcommand{\fq}{\mathfrak{q}}

\newcommand{\gam}{\varGamma} 

\newcommand{\Hom}{\operatorname{Hom}}
\newcommand{\fHom}{\operatorname{\mathcal{H}\!\!\;\mathit{om}}}
\newcommand{\RHom}{\operatorname{{\mathbf R}Hom}}

\newcommand{\Inj}{\operatorname{Inj}}

\newcommand{\kos}[2]{{#1}/\!\!/{#2}}

\newcommand{\KInj}[1]{\mathbf K(\Inj #1)}

\newcommand{\lotimes}{\otimes^{\mathbf L}}

\renewcommand{\mod}{\operatorname{mod}}
\newcommand{\Mod}{\operatorname{Mod}}

\newcommand{\one}{\mathds 1}
\newcommand{\Proj}{\operatorname{Proj}}

\newcommand{\Spec}{\operatorname{Spec}}

\newcommand{\StMod}{\operatorname{StMod}}
\newcommand{\stmod}{\operatorname{stmod}}

\newcommand{\supp}{\operatorname{supp}}

\newcommand{\thick}{\operatorname{thick}}
\newcommand{\Thick}{\operatorname{Thick}}

\newcommand{\phat}{{}^{^\wedge}_\fp}

\newcommand{\iso}{\xrightarrow{\raisebox{-.4ex}[0ex][0ex]{$\scriptstyle{\sim}$}}}
\newcommand{\longiso}{\xrightarrow{\ \raisebox{-.4ex}[0ex][0ex]{$\scriptstyle{\sim}$}\ }}

\title[Local dualisable modular representations]{The spectrum of local dualisable\\ modular representations}

\author[Benson, Iyengar, Krause, and Pevtsova]{Dave Benson, Srikanth
  B. Iyengar, Henning Krause \\ and Julia Pevtsova}

\address{Dave Benson \\ 
Institute of Mathematics\\ 
University of Aberdeen\\ 
King's College\\ 
Aberdeen AB24 3UE\\ 
Scotland U.K.}

\address{Srikanth B. Iyengar\\ 
Department of Mathematics\\
University of Utah\\ 
Salt Lake City, UT 84112\\ 
U.S.A.}

\address{Henning Krause\\ 
Fakult\"at f\"ur Mathematik\\ 
Universit\"at Bielefeld\\ 
33501 Bielefeld\\ 
Germany.}

\address{Julia Pevtsova\\ 
Department of Mathematics\\ 
University of Washington\\ 
Seattle, WA 98195\\ 
U.S.A.}

\begin{document}

\begin{abstract} 
For a point $\mathfrak{p}$ in the spectrum of the cohomology ring of a finite group $G$ over a field $k$, we calculate the spectrum for the subcategory of dualisable objects inside the tensor triangulated category of $\fp$-local and $\fp$-torsion objects in the (big) stable module category of the group algebra $kG$.
\end{abstract}

\keywords{Balmer spectrum, dualisable object, Quillen stratification, lattice of thick tensor ideals, stable module category of a finite group}

\subjclass[2020]{20C20 (primary); 18G80, 20J06 (secondary)}

\date{\today}

\maketitle

\section{Introduction}

Let $G$ be a finite group and $k$ a field of characteristic $p$.  The stable category of $kG$-modules, denoted $\StMod{kG}$, is a tensor triangulated category that has been much studied over the last few decades. The compact objects form the stable category of finitely generated modules $\stmod{kG}$, whose thick tensor ideals were classified by 
Benson, Carlson and Rickard~\cite{Benson/Carlson/Rickard:1997a}.  The localising tensor ideals of $\StMod{kG}$ were later classified by Benson, Iyengar and Krause~\cite{Benson/Iyengar/Krause:2011b}. In both cases, the answers are in terms of the projectivised prime ideal spectrum  of the cohomology ring $\Proj H^*(G,k)$, an object well understood due to the extraordinary work of Quillen~\cite{Quillen:1971b,Quillen:1971c}. Given a homogeneous prime ideal $\fp$ in $\Proj H^*(G,k)$, there is an idempotent functor  $\gam_\fp$ on $\StMod{kG}$ that picks out a minimal localising tensor ideal, to be thought of as a stratum of the stable module category at that prime. Every localising tensor ideal is generated by the minimal ones it contains, and this gives a bijection between the localising tensor ideals of $\StMod{kG}$ and the subsets of $\Proj H^*(G,k)$. For the thick tensor ideals in $\stmod{kG}$, only the specialisation closed subsets are relevant.

In~\cite{Benson/Iyengar/Krause/Pevtsova:2024c} we studied the subcategory $\duals{(\gam_\fp\StMod{kG})}$ of dualisable objects in the tensor  triangulated category $\gam_\fp\StMod{kG}$.  This an essentially small tensor triangulated category and our purpose in this  paper is to describe the thick tensor ideals of $\duals{(\gam_\fp\StMod{kG})}$. The answer is in terms of the spectrum of the completion of the localised cohomology ring, $\Spec H^*(G,k)\phat$. The following is our main theorem.

\begin{theorem}
\label{th:main}
Cohomological support gives a bijection between the thick tensor ideals in
$\duals{(\gam_\fp\StMod{kG})}$ and the specialisation closed subsets of $\Spec H^*(G,k)\phat$.
\end{theorem}

The completion of a local ring does change the spectrum. For  example, let $k$ be a field of odd characteristic and $R=k[x,y]$, $\fp=(x,y)$. Then the polynomial $x^2-y^2(1-y)$ is prime in $R_\fp$, but factorises in $R\phat$ as  $(x+y\sqrt{1-y})(x-y\sqrt{1-y})$. One can mimic this example in the cohomology of an elementary abelian group.

To prove the theorem, we first deal with the case of an elementary abelian $p$-group.
In this case, we use the same route as in~\cite{Benson/Iyengar/Krause:2011b} to transfer the problem into the corresponding problem in a polynomial ring localised at a prime. In that context, we essentially solved the problem in~\cite{Benson/Iyengar/Krause/Pevtsova:2024b}; see also Theorem~\ref{th:dga}. To go from elementary abelian $p$-subgroups to arbitrary finite groups, we use the Quillen stratification, via induction and restriction. Along the way, we also obtain an analogue of the stratification results of Quillen, and of Avrunin and Scott~\cite{Avrunin/Scott:1982a}, for local dualisable modules; see Theorem~\ref{th:stratification}.

\section{Lattice of thick tensor ideals}
\label{se:lattices}
In this section we record a few observations regarding the lattice of thick tensor ideals in a tensor triangulated category.

Let $(\cat C,\otimes,\one)$ be an essentially small tensor triangulated category, with internal function object $\fHom$. An object $x$ is \emph{dualisable} if for each $y\in\cat C $ the natural map
\[
\fHom(x,\one)\otimes y\longrightarrow \fHom(x,y)
\]
is an isomorphism.  We assume that the category $\cat C$ is \emph{rigid}, which means that every object is dualisable. Throughout we also fix  a graded-commutative ring $R$ that acts on $\cat C$ canonically, that is to say, via a map of rings $\phi\colon R\to \End^*_{\cat T}(\one)$. In many cases $R$ is the graded endomorphism ring of $\one$, and $\phi=\mathrm{id}$, but not always.  In what follows we only ever consider homogeneous elements and ideals; a non-homogeneous element has no meaning as an endomorphism of $\one$. In particular,  $\Spec R$ is the set of homogeneous prime ideals of $R$ endowed with the Zariski topology.

\subsection*{Cohomological support}
For a finitely generated ideal $\fa$ of $R$ set
\[
V(\fa)\coloneqq \{\fp\in\Spec R\mid \fa\subseteq\fp\}
\]
and consider the thick subcategory
\[
\cat C_{V(\fa)}\coloneqq \{x\in\cat C\mid \End^*_{\cat C}(x)_\fp=0 \text{ for }
  \fp\in (\Spec R)\setminus V(\fa)\}.
  \]

We recall from \cite{Hovey/Palmieri/Strickland:1997a} the definition of \emph{Koszul objects}. For an element $r$ in $R$ of degree $d$ and an object $x$ in $\cat C$ let $\kos{x}{r}$ denote the cone of the map
$x\xrightarrow{r}\Sigma^{d} x$. For an ideal $\fa=(r_1,\ldots,r_n)$ we set $\kos{x}{\fa}=x_n$, where $x_0=x$ and $x_i=\kos{x_{i-1}}{r_i}$ for $1\leqslant i\leqslant n$. The definition depends on the choice of generators, but the thick subcategory generated by such a Koszul object does not.

\begin{lemma}
  \label{le:koszul}
  We have $\cat C_{V(\fa)}=\thick^\otimes(\kos{\one}{\fa})$.
\end{lemma}
\begin{proof}
  The equality
  $\cat C_{V(\fa)}=\thick(\{\kos{x}{\fa}\mid x\in\cat C\})$ follows from
  \cite[Proposition~3.10]{Benson/Iyengar/Krause:2015a}. It remains
  to note that $\kos{x}{\fa}= \kos{\one}{\fa}\otimes x$.
\end{proof}

A subset $V\subseteq\Spec R$ is said to be a \emph{Thomason} subset if it can be written as $V=\bigcup_i V_i$ such that each $(\Spec R)\setminus V_i$ is Zariski open and quasi-compact; that is to say, $V_i=V(\fa_i)$ for some finitely generated ideal $\fa_i$ in $R$. By definition, the Thomason subsets are the open subsets for the \emph{Hochster dual topology} on $\Spec R$; see \cite{Hochster:1969a}. When $R$ is noetherian,  Thomason subsets are precisely the specialisation closed subsets of $\Spec R$.   

The open subsets of any space
form a \emph{frame}, that is to say,  a complete lattice where arbitrary joins distribute over finite meets. A morphism of frames is a map that preserves arbitrary joins and finite meets; see \cite{Johnstone:1982a} for details. An element $x$ in a complete lattice is \emph{finite} or \emph{compact} if $x\le\bigvee_{i\in I}y_i$ implies
$x\le\bigvee_{i\in J}y_i$ for some finite subset $J\subseteq I$. A frame is \emph{coherent} when every element can be written as a join of finite elements and the finite elements form a sublattice. 

For example, the Thomason subsets of $\Spec R$ form a coherent frame, and a simple calculation shows that the finite elements are precisely the sets of the form $V(\fa)$ where $\fa$ is a finitely generated ideal. Let $\Thick^{\otimes}(\cat C)$ denote the lattice of tensor ideal thick subcategories in $\cat C$; it is also a coherent frame with finite elements thick tensor ideals of the form $\thick^{\otimes}(x)$ for $x\in \cat C$; see \cite[Theorem~3.1.9]{Kock/Pitsch:2017a}. 

For a Thomason subset $V\subseteq \Spec R$ we set
  \[
  \tau_R(V)\coloneqq \bigvee_{V(\fa)\subseteq V} \cat C_{V(\fa)}\,.
\]
In \cite[Proposition~4.9]{Krause:2023a} it is shown that this  induces a map
\begin{equation}
  \label{eq:tau}
  \{\text{Thomason subsets of }\Spec
  R\}\xrightarrow{\ \tau_R\ }\Thick^\otimes(\cat C)\,.
\end{equation}  

For each $x\in \cat C$ the \emph{cohomological support} is by definition
\[
\supp_R(x) \coloneqq \{\fp\in\Spec R\mid \End^*_{\cat C}(x)_\fp\neq 0 \}\,.
\]
We say $\supp_R$ is \emph{finite} if for each $x\in\cat C$ there is a finitely generated ideal $\fa$ of $R$ such that $\supp_R(x)=V(\fa)$. This property holds, for instance, when $R$ is noetherian and the $R$-module $\End^*_\cat C(x)$ is finitely generated for all $x\in\cat C$. When $\supp_R$ is finite one gets a frame morphism
\begin{equation}
\label{eq:class-tt}
\supp_R\colon \Thick^\otimes(\cat C)\xrightarrow{\ \supp_R\ }\{\text{Thomason subsets of }\Spec R\}
\end{equation}
that restricts to a map on the sublattice of finite elements.  More is true.

\begin{lemma}
  The cohomological support $\supp_R$ is finite if and only if there is an adjoint
pair of maps between posets
\[
\begin{tikzcd}
\Thick^\otimes(\cat C) \ar[rr,yshift=2.5,"\supp_R"] &&  \ar[ll,yshift=-2.5,"\tau_R"]
\{\text{Thomason subsets of }\Spec R\}\,;
\end{tikzcd}
\]
that is to say, for $\cat D\subseteq\cat C$ and $V\subseteq\Spec R$ one has
\[
\supp_R(\cat D)\subseteq V\quad\iff\quad\cat D\subseteq\tau_R(V)\,.
\]
\end{lemma}
\begin{proof}
  Suppose  $\supp_R$ is finite, so that one gets a well-defined map from the frame of thick tensor ideals of $\cat C$ to the frame of Thomason subsets of $\Spec R$.  Let $V$ be a Thomason subset of $\Spec R$ and
  $x\in \cat C$. Then we have
\begin{align*}
\supp_R(x)\subseteq V&\iff \supp_R(x)\subseteq V(\fa) \text{ for some }V(\fa)\subseteq V\\
                     &\iff x\in\cat C_{V(\fa)}\text{ for some }V(\fa)\subseteq V\\
                     &\iff  x\in\tau_R(V).
\end{align*}
The first equivalence holds because $\supp_R(x)$ is finite as an element in the frame of Thomason subsets of $\Spec R$, by hypothesis. The other ones are immediate from the definitions.

Now suppose ($\supp_R,\tau_R$) form an adjoint pair. This implies that
$\supp_R(x)$ is Thomason for each   $x\in \cat C$. A standard argument
shows that the left adjoint preserves finiteness if the right
adjoint  preserves all joins. It remains to recall that $\tau_R$ preserves joins.
\end{proof}

A consequence of the preceding result is that $\supp_R$ is an isomorphism if and only if $\tau_R$ is an isomorphism. This can be characterised as follows.

\begin{lemma}
\label{le:hopkins}
  Cohomological support induces an isomorphism
  \begin{equation*}
     \Thick^\otimes(\cat C)\longrightarrow\{\text{Thomason subsets of }\Spec
     R\}\,.
\end{equation*}  
  if and only if the following hold:
  \begin{enumerate}[\quad\rm(1)]
  \item $\supp_R$ is finite;
  \item $\supp_R(\cat C)=\Spec R$;
\item 
$\supp_R(y)\subseteq\supp_R(x)$ implies $\thick^\otimes(y)\subseteq\thick^\otimes(x)$
for all $x,y\in\cat C$.
\end{enumerate}
\end{lemma}

\begin{proof}
The forward direction is clear, since the finite elements in $\Thick^\otimes(\cat C)$ are precisely the subcategories of the form  $\thick^\otimes(x)$ for some $x\in\cat C$. Conversely, suppose conditions (1)--(3) hold.  Then $\tau_R\supp_R(\thick^\otimes(x))=\thick^\otimes(x)$ holds and hence
$\tau_R\supp_R=\mathrm{id}$. On the other hand, $\supp_R$ is surjective since $\supp_R(\kos{\one}{\fa})=V(\fa)$ for any finitely generated ideal $\fa$ of $R$.   
\end{proof}

\begin{remark}\label{re:hopkins}
    There is an analogue of Lemma~\ref{le:hopkins} for the lattice of all thick subcategories, with essentially same proof. The tensor unit is then replaced by a generator of $\cat C$, which does exist when $\supp_R$ is finite and $\supp_R(\cat C)=\Spec R$.
\end{remark}

\subsection*{Tensor triangulated support}
Following \cite{Balmer:2005a} we consider the \emph{Balmer spectrum}  $\bSpec\cat C$ of $\cat C$,
and for any $x\in\cat C$ set
\[
\supp_\otimes(x) \coloneqq\{\cat P\in \bSpec\cat C\mid x\not\in\cat P\}\,.
\]
By \cite[Theorem~4.10]{Balmer:2005a} there are mutually inverse isomorphisms of frames
\begin{equation}
\begin{tikzcd}
\Thick^{\otimes}(\cat C) \ar[rr,yshift=2.5,"\supp_\otimes"] &&  \ar[ll,yshift=-2.5,"\tau_\otimes"]
\{\text{Thomason subsets of }\bSpec\cat C\}
\end{tikzcd}
\end{equation}
where for each pair $\cat D\subseteq\cat C$ and $V\subseteq\bSpec\cat C$ one has
\[
\supp_\otimes(\cat D)\coloneqq\bigcup_{x\in\cat D}\supp_\otimes(x)\quad\text{and}\quad
  \tau_\otimes(V)\coloneqq\{x\in\cat C\mid\supp(x)\subseteq V\}\,.
  \]
The Balmer spectrum is related to the spectrum of the ring $R$ that acts on $\cat C$, via the \emph{comparison map}
\begin{equation}
\label{eq:bs-zs}
\bSpec\cat C\xrightarrow{\ \rho_{\cat C}\ } \Spec \End^*_{\cat C}(\one)\xrightarrow{\ \Spec\phi\ } \Spec R
\end{equation}
where $\rho_{\cat C}$ takes a prime ideal $\cat P\subseteq\cat C$ to the preimage of the unique maximal ideal of $\End^*_{\cat C/\cat P}(\one)$ under the canonical homomorphism
$\End^*_{\cat C}(\one)\to \End^*_{\cat C/\cat P}(\one)$; see \cite{Balmer:2010b}. We need a simple lemma.

\begin{lemma}\label{le:comparison}
For $r\in\End^*_\cat C(\one)$ we have $\rho^{-1}_\cat C(V(r))=\supp_\otimes(\kos{\one}{r})$.
\end{lemma}
\begin{proof}
  Let $\cat P\in\bSpec\cat C$. Then $\cat
  P\in\supp_\otimes(\kos{\one}{r})$ if and only if the morphism
  $\one\xrightarrow{r}\Sigma^{|r|}\one$ is not invertible in $\cat C/\cat P$. This
  means $r\in\rho_\cat C(\cat P)$, so $\cat P\in \rho^{-1}_\cat C(V(r))$.
\end{proof}

The following result provides a comparison between cohomological and
tensor triangulated support.

\begin{proposition}
\label{pr:bs-lattices}
  Let $\cat C$ be an essentially small rigid tensor triangulated
  category. Then the comparison map \eqref{eq:bs-zs} induces a
  commutative diagram:
\begin{equation*}
  \begin{tikzcd}
    \{\text{Thomason subsets of }\Spec
    R\} \ar[r,"\tau_R"]\ar[d,"\rho^*_{\cat C}(\Spec\phi)^*"] &\Thick^{\otimes}(\cat C)\ar[d,equal]\\
    \{\text{Thomason subsets of }\bSpec\cat C\} \ar[r,"\tau_\otimes", "\sim" swap]&\Thick^{\otimes}(\cat C)
\end{tikzcd}
\end{equation*}
In particular, the comparison map\eqref{eq:bs-zs} is a homeomorphism if and only if the
  map \eqref{eq:tau} is an isomorphism.
\end{proposition}

\begin{proof}
The map $\rho^*_{\cat C}$ is given by taking preimages along
$\rho_{\cat C}$, and $(\Spec\phi)^*$ is defined analogously.  We need to check that these maps are well-defined and that they make the diagram commutative. First observe that for any $r\in R$ the preimage of $V(r)$ under $\Spec\phi$ equals   $V(\phi(r))$, and its preimage  under $\rho_{\cat C}$ equals $\supp_\otimes(\kos{\one}{r})$ by Lemma~\ref{le:comparison}. Thus
  \[
  (\rho^*_{\cat C}(\Spec\phi)^*)(V(\fa))=\supp_\otimes(\kos{\one}{\fa})
  \] 
for any finitely generated ideal $\fa$ of $R$. Using Lemma~\ref{le:koszul} this yields the commutativity, keeping in mind that all maps in the diagram preserve arbitrary joins.  The diagram implies that $\tau_R$ is a bijection if and only if $\rho^*_{\cat C}(\Spec\phi)^*$ is a bijection. The spaces in question are sober, so are determined up to a homeomorphism by their frames of open subsets \cite{Johnstone:1982a}. Thus $\rho^*_{\cat C}(\Spec\phi)^*$ is a bijection if and only if $(\Spec\phi)\rho_{\cat C}$ is a homeomorphism.
\end{proof}

\begin{remark}
Suppose that the cohomological support is finite. Then the support map
$\supp_R$ from \eqref{eq:class-tt} is a bijection if and only if the comparison map \eqref{eq:bs-zs} is a homeomorphism. This follows from the adjointness of the pair $(\supp_R,\tau_R)$.
\end{remark}

\subsection*{Change of categories}
In Section~\ref{se:eab} we have to deal with changes of categories. In that context, the tensor structure on categories turns out to be not essential since all thick subcategories are tensor ideals. So it is convenient not to have to worry whether functors are monoidal. 

Consider an essentially small tensor triangulated category $\cat C$ with central $R$-action and the lattice $\Thick(\cat C)$ of all thick subcategories, rather than tensor ideal thick subcategories of $\cat C$. This need not be a frame. The cohomological support function $\supp_R$ is defined just as above and gives a map from $\Thick(\cat C)$ to the Thomason subsets of $\Spec R$ when $\supp_R$ is finite. 

\begin{lemma}
\label{le:hopkins-cat}
Let $f\colon\cat C\to\cat D$ and $g\colon\cat D\to\cat C$ be an adjoint pair of $R$-linear exact functors between $R$-linear triangulated categories. Suppose that $\thick(gf(x))=\thick(x)$ for all $x\in\cat C$.  If taking support induces a lattice isomorphism
\[  
\Thick(\cat D)\xrightarrow{\ \supp_R\ } \{\text{Thomason subsets of }\Spec R\}\,,
\]
then  taking support induces a lattice isomorphism
\[  \Thick(\cat C)\xrightarrow{\ \supp_R\ } \{\text{Thomason subsets of }\Spec R\}\,.
\]
\end{lemma}
\begin{proof}
First observe that $\supp_R(f(x))=\supp_R(x)$ for each $x\in\cat C$, since there is an $R$-linear isomorphism
\[
\End^*_{\cat D}(f(x))\cong\Hom^*_{\cat C}(x,gf(x))
\]
and $\thick(gf(x))=\thick(x)$. Now we check the conditions in Lemma~\ref{le:hopkins} for $\cat C$, keeping in mind Remark~\ref{re:hopkins}. 
Conditions (1) and (2) follow from our first observation, because they hold for $\cat D$. For (3) we have
\begin{align*}
    \supp_R(x)\subseteq\supp_R(y)&\implies  \supp_R(f(x))\subseteq\supp_R(f(y))\\
    &\implies \thick(f(x))\subseteq\thick(f(y))\\
    &\implies \thick(gf(x))\subseteq\thick(gf(y))\\
    &\implies \thick(x)\subseteq\thick(y)
\end{align*}
and this finishes the proof.
\end{proof}

\section{Stable module category}
\label{se:stmod}
In this section we recall a characterisation of local dualisable objects in the stable module category of a finite group, from  \cite{Benson/Iyengar/Krause/Pevtsova:2024c}.

Let $G$ be a finite group and  $k$ a field of positive characteristic $p$ dividing $|G|$. In the rest of this manuscript the focus is on the stable category of (all) $kG$-modules
\[
\cat T_G\coloneqq \StMod kG\,.
\]
This is a rigidly compactly generated tensor triangulated category, with product $-\otimes_k-$ with diagonal $G$-action, and unit $k$. Let $H^*(G,k)$ denote the cohomology algebra of $G$. There is a natural map of graded $k$-algebras $H^*(G,k) \to \End^*_{\cat T_G}(k)$, which provides a canonical $H^*(G,k)$-action on $\cat T_G$.

Fix $\fp$ in $\Proj H^*(G,k)$ and let  $\gam_\fp \colon \cat T_G\to \cat T_G$ be the idempotent functor corresponding to $\fp$; see, for instance, \cite[Section~5]{Benson/Iyengar/Krause:2008a}. Its image $\gam_\fp\cat T_G$ is the full subcategory of $\cat T_G$ consisting of $\fp$-local and $\fp$-torsion objects, and that is itself a compactly generated tensor triangulated category with product inherited from $\cat T_G$ and unit  $\gam_{\fp}k$. We also know the graded endomorphism ring of the unit: The natural map $H^*(G,k)\to \End^*_{\cat T_G}(\gam_\fp k)$ induces an isomorphism
\begin{equation}
\label{eq:benson-con}
 H^*(G,k)\phat \longiso \End^*_{\cat T_G}(\gam_\fp k)\,.
\end{equation}
This is proved by Benson and Greenlees~\cite[Theorem~2.6]{Benson/Greenlees:2008a}; see also \cite[Theorem~2]{Benson:2008a}. The following result is extracted from \cite[Theorem~1.1]{Benson/Iyengar/Krause/Pevtsova:2024c}.

\begin{theorem}
\label{th:regular-paper}
Fix $\fp$ in $\Proj H^*(G,k)$. There are equalities
\[
\duals{(\gam_{\fp}\cat T_G)} = \thick(\{\gam_{\fp}M\mid M\in\stmod kG\}) = \thick(\gam_{\fp}C)\,,
\]
where $C$ is a generator for the triangulated category $\stmod kG$. \qed
\end{theorem}

The result below implies that cohomological support, with respect to $R$, is finite on the  category of local dualisable objects, so the results of Section~\ref{se:lattices} apply to it.

\begin{corollary}
\label{cor:noeth} 
Fix $\fp$ in $\Proj H^*(G,k)$ and set $\cat C=\duals{(\gam_\fp\cat T_G)}$. Then the $k$-algebra  $R=\End^*_{\cat C}(\gam_\fp k)$ is noetherian and  for all $M,N$ in $\cat C$ the $R$-module $\Hom^*_{\cat C}(M,N)$ is finite. In particular, cohomological support, $\supp_R$, is finite on $\cat C$. 
\end{corollary}

\begin{proof}
    Set $S=H^*(G,k)$. Since this ring is noetherian, and localisations and completions preserve the noetherian property, it follows from \eqref{eq:benson-con} that the ring $R$ is noetherian. With $C$ as in Theorem~\ref{th:regular-paper}, and $C^\vee=\Hom_k(C,k)$, one has isomorphisms of $R$-modules
    \begin{align*}
    \Hom^*_{\cat T_G}(M,\gam_\fp C)
        &\cong \Hom^*_{\cat T_G}(M,C\otimes_k \gam_\fp k)\\
        &\cong \Hom^*_{\cat T_G}(M\otimes_k C^\vee, \gam_\fp k)\\     
        &\cong \Hom^*_S(\Hom^*_{\cat T_G}(k,\Omega^{d}(M\otimes_k C^\vee)), I(\fp)) \\
        &\cong \Hom^*_{S}(\Hom^*_{\cat T_G}(k,M\otimes_k\Omega^{d}(C^\vee)), I(\fp)) \\
        &\cong \Hom^*_R(\Hom^*_{\cat T_G}(k,M\otimes_k \Omega^{d}(C^\vee)), I(\fp))
    \end{align*}
    Here $I(\fp)$ is the injective hull of the $S$-module $S/\fp$ and $d$ is the Krull dimension of $S/\fp$. The third isomorphism is by  \cite[Theorem~5.1]{Benson/Iyengar/Krause/Pevtsova:2019a}; see also \cite[Theorem~2.6]{Benson/Greenlees:2008a}. The last isomorphism holds because the modules $M\otimes_k \Omega^{d}(C^\vee)$ and $I(\fp)$ are $\fp$-local and $\fp$-torsion. The remaining isomorphisms are standard. Since $M$ is dualisable in  $\gam_\fp\cat T_G$ and the module $C$, and hence also $\Omega^{d}(C^\vee)$, is compact in $\cat T_G$, the module $M\otimes_k \Omega^{d}(C^\vee)$ is dualisable in $\gam_\fp\cat T_G$. Thus the $R$-module $\Hom^*_{\cat T_G}(k,M\otimes_k\Omega^{d}(C^\vee))$ is artinian, by \cite[Theorem~8.5]{Benson/Iyengar/Krause/Pevtsova:2024c}. Using Matlis duality and the isomorphisms above we deduce that $\Hom^*_{\cat T_G}(M,\gam_\fp C)$, and, hence, $\Hom^*_{\cat C}(M,\gam_\fp C)$ is noetherian.  Since $N$ is in $\thick(\gam_{\fp}C)$, again by Theorem~\ref{th:regular-paper}, we deduce that the $R$-module $\Hom^*_{\cat C}(M,N)$ is noetherian; equivalently, finitely generated.
\end{proof}
 
Theorem~\ref{th:regular-paper} also allows us to translate the problem of classifying the tensor ideal thick categories of $\duals{(\gam_{\fp}\cat T_G)}$ to the corresponding classification problem for $\thick(\gam_\fp C)$. There is one further change in perspective that is helpful.

\subsection*{The homotopy category of injective modules}
Let $\KInj{kG}$ be the homotopy category of complexes of injective $kG$-modules; this too is a rigidly compactly generated tensor triangulated category, with product $-\otimes_k-$ with diagonal $G$-action and unit $\bfi k$, the injective resolution of the trivial $kG$-module $k$. The graded endomorphism ring of $\bfi k$ is precisely $H^*(G,k)$.  The subcategory of compact, equivalently, dualisable, objects in $\KInj{kG}$ naturally identifies with $\dbcat{kG}$, the bounded derived category of finite generated $kG$-modules. This equivalence assigns a complex $M\in \dbcat{kG}$ to $\bfi M$.

The stable module category, $\cat T_G$,  can be identified as a tensor triangulated category with the full subcategory of acyclic complexes in $\KInj{kG}$, and one has a recollement
\[
\begin{tikzcd}[column sep = huge]
    \cat T_G = \StMod{kG} \arrow[tail]{r} 
    	&\KInj{kG} \arrow[twoheadrightarrow,swap,yshift=1.5ex]{l}    \arrow[twoheadrightarrow,yshift=-1.5ex]{l}
		    \arrow[twoheadrightarrow]{r} 
        &\dcat{\Mod kG}\arrow[tail,swap,yshift=1.5ex]{l}\arrow[tail,yshift=-1.5ex]{l}
  \end{tikzcd}
\]
For each $\fp$ in $\Proj H^*(G,k)$  applying $\gam_\fp$ to the inclusion $\cat T_G\subset \KInj{kG}$ yields an equivalence of tensor triangulated categories:
\begin{equation*}
\label{eq:stmod-kinj}
\gam_{\fp}\cat T_G\longiso\gam_{\fp}\KInj{kG}\,.
\end{equation*}
See, for instance, \cite[Section~8]{Benson/Iyengar/Krause/Pevtsova:2024c}. Thus Theorem~\ref{th:regular-paper} translates to an equality
\begin{equation*}
\duals{(\gam_{\fp}\KInj{kG})} = \thick(\{\gam_{\fp}(\bfi M)\mid M\in\mod kG\})\,.    
\end{equation*}
Thus the classification problem that we have set out to solve can be recast in terms of local dualisable objects in $\KInj{kG}$. This is important in Section~\ref{se:eab}. 

Having made this transition, a natural question is to compute the spectrum of local dualisable objects in $\KInj{kG}$ also at  $\fm=H^{\geqslant 1}(G,k)$, the unique closed point of $\Spec H^*(G,k)$; this does not arise for the stable module category because $\gam_\fm \cat T_G = \{0\}$.  The recollement above induces an equivalence 
\[
\gam_\fm \KInj{kG}\longiso \dcat{\Mod kG}
\]
of tensor triangulated categories, leading to the equivalence
\[
\duals{(\gam_{\fm}\KInj{kG})} \longiso \dbcat{kG}\,.
\]
Keep in mind that the tensor structure on $\dcat{\Mod kG}$ is $-\otimes_k-$ with diagonal $G$-action. The main result of \cite{Benson/Carlson/Rickard:1997a} computes the spectrum of the category on the right, yielding homeomorphisms:
\begin{equation*}
\label{eq:spec-maps}
\bSpec \duals{(\gam_\fm\KInj{kG})} \xrightarrow{\ \rho_G\ } \Spec \End^*(k) \longiso \Spec H^*(G,k)\,.
\end{equation*}
Theorem~\ref{th:spec-G} below is an extension of this result to all primes $\fp$ in $\Spec H^*(G,k)$.

\section{Elementary abelian groups} 
\label{se:eab}
Let $p$ be a prime number, $k$ a field of characteristic $p$, and $E=(\bbZ/p)^r$ the elementary abelian $p$-group of rank $r$. Fix  $\fp$ in $\Proj H^*(E,k)$ and set
\[
\cat C\coloneqq \duals{(\gam_\fp\StMod kE)}\quad\text{and} \quad
    R\coloneqq \End^*_{\cat C}(\gam_\fp k)\cong H^*(E,k) \phat\,.
    \]
As noted in the last section $\cat C$ is an essentially small tensor-triangulated category, with product inherited from $\StMod kE$ and unit $\gam_\fp k$. The isomorphism is from \eqref{eq:benson-con}. The result below computes the spectrum of $\cat C$.

\begin{theorem}
\label{th:spec-E}
Taking cohomological support induces a lattice isomorphism
\[
   \Thick^{\otimes}(\cat C) \xrightarrow{\ \supp_R \ } \{\text{Thomason subsets of }\Spec R \} \,.
\]
\end{theorem}

The proof takes some preparation. A key ingredient is the Bernstein--Bernstein--Gelfand (BGG) correspondence, as formulated in \cite{Benson/Iyengar/Krause:2011b}---see also the proof of \cite[Theorem~8.2]{Benson/Iyengar/Krause/Pevtsova:2024a}. We begin by recalling the various categories that appear in it, keeping the notation from \cite{Benson/Iyengar/Krause/Pevtsova:2024a}.

 Let $S$ be the $k$-subalgebra of $H^*(E,k)$ generated by the image of the Bockstein map $H^1(E,k)\to H^2(E,k)$. This is a polynomial algebra over $k$ on $r$ generators and the map $S\subset H^*(E,k)$ is finite. Set $A=kE$. This is an artinian local ring, isomorphic to $k[z_1,\dots,z_r]/(z_1^p,\dots,z_r^p)$, where $r$ is the rank of $E$. 
 
 The BGG correspondence we use is formulated in terms of appropriate homotopy categories, so let  $\KInj A$ denote the homotopy category of injective $A$-modules  viewed as an $H^*(E,k)$-linear category; see Section~\ref{se:stmod}. By restriction, it is also an $S$-linear category.

 Let $B$ be the Koszul complex of $A$, viewed as a dg $A$-algebra. Let $\Lambda$ be the exterior algebra on $r$ generators of (upper) degree 1, viewed as a dg $k$-algebra with zero differential. One has the following $S$-linear compactly generated categories and $S$-linear functors relating them:
\begin{equation}
\label{eq:bgg}
\begin{tikzcd}
    \dcat S \arrow[r,leftrightarrow,"\sim"] &
    \KInj \Lambda \arrow[r,leftrightarrow, "\sim"] &
    \KInj B \arrow[r,yshift= -.75ex,"g" swap] &
    \KInj A \arrow[l,yshift=.75ex,"f" swap]
\end{tikzcd}
\end{equation}
Here $\dcat S$ is the derived category of dg modules over $S$, viewed as a dg algebra with zero differential. The functor $g$ is restriction along the map $A\to B$ of dg algebras,
and $f=B\otimes_A-$ is its left adjoint; it is exact because $B$ is finite-free as an $A$-complex. The equivalence on the left is from \cite[Theorem~6.2]{Benson/Iyengar/Krause:2011b} and the one in the middle follows from \cite[Proposition~4.6]{Benson/Iyengar/Krause:2011b} applied to the map $\varphi$ in \cite[Lemma~7.1]{Benson/Iyengar/Krause:2011b}. The functors $f,g$, though not equivalences, have the   following properties that suffice for the present purpose.

 \begin{lemma}
\label{le:AB}
The functors $f,g$ preserve coproducts and products, and 
 \[
 \thick(gf(X))=\thick(X)\quad\text{for all $X\in \KInj{A}$.}
 \]
\end{lemma}

\begin{proof}
It is clear that $f,g$ preserve coproducts and products. To verify the equality of thick subcategories, we consider $\cat P={\mathbf K}^b(\mathrm{proj}\, A)$, the homotopy category of finite projective (equivalently, free) complexes of $A$-modules. It is a tensor triangulated category with product $-\otimes_A-$ and unit $A$. Since $\End^*_{\cat P}(A)=A$, the ring $A$ has a canonical action on $\cat P$. The $A$-complex $B$ is a Koszul object for the maximal ideal of $A$, and $A$ itself can be viewed as a Koszul object for the zero ideal. Since the maximal ideal of $A$ is nilpotent, one has $\thick(B)=\thick(A)$ as subcategories of $\cat P$; see~\cite[Lemma~6.0.9]{Hovey/Palmieri/Strickland:1997a}.

The assignment $X\mapsto P\otimes_AX$ for $P\in \cat P$ and $X\in \KInj A$ defines an action of $\cat P$ on $\KInj A$, in the sense of Stevenson~\cite{Stevenson:2013a}.
Applying $-\otimes_AX$ to the equality $\thick(B)=\thick(A)$ yields that $\thick(B\otimes_AX)=\thick(X)$ as subcategories of $\KInj{A}$, as desired.
 \end{proof}

\begin{proof}[Proof of Theorem~\ref{th:spec-E}]
We keep the notation introduced above. Since $A=kE$ is a local ring the trivial $kE$-module $k$ generates $\stmod kE$. It follows that the unit $\gam_\fp k$ of $\cat C$ generates the category; see Theorem~\ref{th:regular-paper}. Thus every thick subcategory of $\cat C$ is tensor ideal, and $\Thick^{\otimes}(\cat C)=\Thick(\cat C)$. In particular, the tensor structure on the category is no longer critical. This is important because the argument below uses the BGG correspondence and the functors that appear in it are not all monoidal.

 Recall the inclusion of $k$-algebras $S\subset H^*(E,k)$. The induced map 
 \[
 \Spec H^*(E,k) \to \Spec S
 \]
 is a homeomorphism.  Moreover identifying $\fp$ with its image in $\Spec S$, the natural map $S\phat \to H^*(E,k) \phat =R$ induces a homeomorphism 
\[
\Spec R \longiso\Spec S\phat
\]
that identifies $\supp_{R}(X)$ with $\supp_{S\phat}(X)$ for each $X$ in $\cat C$. In particular, one gets that $\supp_S(\gam_{\fp}k)=\Spec S\phat$.
Therefore, in view of Lemma~\ref{le:hopkins}, our task is to  verify that for objects $X,Y$ in $\cat C$, the following statement holds:
\begin{equation}
    \label{eq:dcat-hopkins}
    \supp_{S\phat}(X)\subseteq \supp_{S\phat}(Y)\quad\implies\quad  \thick(X)\subseteq\thick(Y)\,.
\end{equation} 
 From now on, we view $\cat C$ as an $S\phat $-linear category via restriction along $S\phat \to R$.  

Let $\cat C_A$ and $\cat C_B$ denote the subcategory of compact objects in $\KInj A$ and $\KInj B$, respectively; the compact objects in $\dcat S$ coincide with  $\thick(S)$. Since $f,g$ preserve coproducts, and the other functors in \eqref{eq:bgg} are equivalences, they all take compact objects to compact objects. Because these functors are also $S$-linear, applying $\gam_\fp$ to the diagram \eqref{eq:bgg} yields the diagram of $S\phat$-linear functors:
\[
\begin{tikzcd}
    \thick(\gam_\fp S) \arrow[r,leftrightarrow,"\sim"] &
    \thick(\gam_\fp(\cat C_B)) \arrow[r,yshift= -.75ex,"g" swap] &
    \thick(\gam_\fp(\cat C_A)) = \cat C \arrow[l,yshift=.75ex,"f" swap]
\end{tikzcd}
\]
Our task is to verify that property \eqref{eq:dcat-hopkins} holds for objects in $\cat C$. It follows from Theorem~\ref{th:dga} below that this property holds for objects in  $\thick(\gam_\fp S)$ and hence, by the equivalence above, also  for objects in $\thick(\gam_\fp(\cat C_B))$.  It remains to deduce that this property descends to $\cat C$, by Lemmas~\ref{le:hopkins-cat} and \ref{le:AB}.
\end{proof}

\subsection*{Formal dg algebras}
Let $S$ be a noetherian, graded-commutative algebra, viewed as a dg algebra with zero differential, and $\dcat S$ the derived category of dg $S$-modules. 
Then $\dcat S$ is a rigidly compactly generated tensor triangulated category, with product $-\lotimes_S-$ and unit $S$.

 \begin{theorem}
 \label{th:dga}
Fix $\fp\in \Spec S$ and set $\cat C=\thick(\gam_\fp S)$, viewed as a subcategory of $\dcat S$.  The natural map $S\to \End^*_{\cat C}(\gam_\fp S)$ induces an isomorphism $S\phat \cong \End^*_{\cat C}(\gam_\fp S)$, and taking cohomological support induces a lattice isomorphism
\[
   \Thick(\cat C) \xrightarrow{\ \supp_R \ } \{\text{Thomason subsets of }\Spec S\phat \} \,.
\]
 \end{theorem}
 
 \begin{proof}
 One has an isomorphism $\gam_\fp S\cong \gam_\fp(S_\fp)$; thus, replacing $S$ by $S_\fp$ we can assume that $S$ is local and that $\fp=\fm$ is the maximal ideal of $S$. Let $\widehat S$ denote the $\fm$-adic  completion of $S$ with respect to $\fm$. Derived Morita theory yields an  adjoint pairs of $\widehat S$-linear triangle equivalences 
\[
\begin{tikzcd}[column sep=large]
\thick_{\widehat S}(\widehat S) \arrow[rr, rightarrow, yshift=.75ex, "{\gam_{\fm}S\lotimes_{\widehat S}-}"] 
	&& \arrow[ll, rightarrow,yshift=-.75ex, "{\RHom_S(\gam_{\fm}S,-)}"]
		\thick_S(\gam_\fm S) \,.
\end{tikzcd}
\]
It thus suffices to prove that the lattice of thick subcategories of $\thick(\widehat S)\subset \dcat{\widehat S}$ is classified by $\Spec \widehat S$. 
This thick subcategory consists precisely of the perfect dg $\widehat S$-modules, by, for instance, \cite[Theorem~4.2]{Avramov/Buchweitz/Iyengar/Miller:2010a}, and $\widehat S$ is a noetherian  dg algebra with differential zero. It follows from \cite[Theorem~3.2]{Carlson/Iyengar:2015a} that taking cohomological support gives a lattice isomorphism between the lattice of thick subcategories of $\thick_{\widehat S}(\widehat S)$ and the Thomason subsets of $\Spec \widehat S$. To conclude the proof of the theorem, it remains to apply Lemma~\ref{le:hopkins-cat}.
 \end{proof}

\section{Finite groups}
In this section we extend Theorem~\ref{th:spec-E} to arbitrary finite groups, thereby justifying Theorem~\ref{th:main}. Throughout $G$ is a finite group,  $k$ a field of positive characteristic $p$ dividing $|G|$, and $\fp$ a homogeneous, non-maximal, prime ideal in $H^*(G,k)$, the cohomology algebra of $G$. Consider the category of local dualisable objects at $\fp$:
\[
\cat C\coloneqq \duals{(\gam_\fp\StMod(kG))} \quad\text{and}\quad R\coloneqq \End^*_{\cat C}(\gam_\fp k)\,.
\]
The result below is a reformulation of Theorem~\ref{th:main}.

\begin{theorem}
\label{th:spec-G}
The  map $H^*(G,k) \to R$ induces an isomorphism  $H^*(G,k) \phat \cong R$ of rings, and taking cohomological support induces a lattice isomorphism
\[
   \Thick^{\otimes}(\cat C) \xrightarrow{\ \supp_R \ } \{\text{Thomason subsets of }\Spec R \} \,.
\]
\end{theorem}

The proof is by a reduction to the case of elementary abelian subgroups, which is settled by Theorem~\ref{th:spec-E}. It takes up the entire section. Throughout we keep  the notation as in the statement of the theorem. 

In what follows for any subgroup $E$ (the focus is on elementary abelian subgroups), we write $\cat T_E$ instead of $\StMod kE$, for ease of notion.  Let $\cat A(G)$ be the category, introduced by Quillen~\cite[\S 5]{Quillen:1971a}, whose objects are the elementary abelian $p$-subgroups of $G$, and  the morphisms from $E$ to $E'$ consist of group homomorphisms induced by conjugation in $G$. The proof of the theorem consists of constructing the commutative diagram of topological spaces
\begin{equation}
\label{eq:colimsquare}
    \begin{tikzcd}
\underset{E\in \cat A(G)}{\colim} \bSpec \duals{(\gam_{\fp}{\cat T}_E)} \ar[d,twoheadrightarrow, "(1)" swap] \ar[r,"\sim" swap, "(2)"] 
    & \underset{E\in \cat A(G)}{\colim} \Spec H^*(E,k)\phat \ar[d, "\sim" swap, "(3)"] \\
 \bSpec \duals{(\gam_{\fp}{\cat T}_G)} \ar[r,twoheadrightarrow,"(4)" swap] &  \Spec H^*(G,k)\phat \,.
\end{tikzcd}
\end{equation}
and justifying that (1) and (4) are surjections, while (2) and (3) are homeomorphisms. Once this is done, it is clear that (4) must be a homeomorphism as well, which is the desired result; see Proposition~\ref{pr:bs-lattices}.

Setting up the commutative diagram already takes some preparation.  To start with, for any subgroup $E\leq G$ restriction induces a tensor triangulated functor
\[
\rho_E\colon \cat T_G \longrightarrow \cat T_E
\]
and a map of $k$-algebras $\rho^*_E \colon H^*(G,k)\longrightarrow H^*(E,k)$. Hence one gets an induced map on Zariski spectra:
\[
\Spec \rho^*_E \colon \Spec H^*(E,k)\longrightarrow  \Spec H^*(G,k)\,.
\]
We say that \emph{ $\fp$ is supported on $E$} to mean that $\fp$ is in the image of the map above. Such $E$ are the only relevant elementary abelian groups for the present purposes.

Viewing $\cat T_E$ as an $H^*(G,k)$-linear category via $\rho^*_E$, consider the subcategory $\gam_{\fp}(\cat T_E)$ of $\fp$-local and $\fp$-torsion modules. This category is not trivial (that is to say, not equivalent to $\{0\}$) if, and only if, $\fp$ is supported on $E$; see~\cite[Theorem~11.2]{Benson/Iyengar/Krause:2011b}. By \cite[Corollary~7.8]{Benson/Iyengar/Krause:2012b}, the functor $\rho_E$ restricts to the subcategory of $\fp$-local and $\fp$-torsion objects:
\[
\rho_E\colon \gam_{\fp}\cat T_G\longrightarrow \gam_{\fp}\cat T_E\,.
\]
Since it is a tensor triangulated functor, the natural map~\eqref{eq:bs-zs} gives a commutative diagram of continuous maps
\begin{equation}
\label{eq:square}
    \begin{tikzcd}
\bSpec \duals{(\gam_{\fp}{\cat T}_E)} \ar[d,"(1)" swap] \ar[r, "(2)"] 
    & \Spec H^*(E,k)\phat \ar[d, "(3)"] \\
 \bSpec \duals{(\gam_{\fp}{\cat T}_G)} \ar[r,twoheadrightarrow,"(4)" swap] &  \Spec H^*(G,k)\phat \,.
\end{tikzcd}
\end{equation}
By \cite[Theorem~7.3]{Balmer:2010b}, the map (4) is onto because the ring $H^*(G,k)\phat$ is noetherian. The map (3) is induced by the map
\[
H^*(G,k)\phat\longrightarrow H^*(E,k)\phat
\]
obtained from $\rho^*_E$ by localising at $\fp$ and then completing. The next result adds further information to the diagram.

\begin{lemma}
\label{le:map2}
 Let $E$ be an elementary abelian $p$-group on which $\fp$ is supported.  The composition of natural maps of rings
\[
H^*(E,k)\longrightarrow \End^*_{\cat T_E}(k)\longrightarrow  \End^*_{\cat T_E}(\gam_{\fp} k)
\]
induces an isomorphism  $H^*(E,k)\phat \iso \End^*_{\cat T_E}(\gam_{\fp} k)$. In particular, $\End^*_{\cat T_E}(\gam_{\fp}k)$ is a noetherian ring. Moreover,  the map \emph{(2)} in diagram \eqref{eq:square} is a homeomorphism.
\end{lemma}

\begin{proof}
The map $\rho^*_E$ is finite. Let $\fq_1, \ldots, \fq_n  \in \Spec H^*(E,k)$ be  the primes lying over $\fp \in \Spec H^*(G,k)$ with respect to $\rho^*_E$. For any $X \in \cat T_E$, we have 
\[
\gam_\fp X = \bigoplus_{i=1}^n \gam_{\fq_i} X\,,
\]
and therefore $\duals{(\gam_\fp \cat T_E)} =  \prod_{i=1}^n \duals{(\gam_{\fq_i} \cat T_E)}$. The finiteness of $\rho^*_E$ also yields that $\fq_i\notin V(\fq_j)$ for $i\ne j$, and hence we get from \cite[Theorem~11.13]{Benson/Iyengar/Krause:2011b} that
\[
\Hom^*_{\cat T_E}(\gam_{\fq_i}k,\gam_{\fq_j}k) = 0 \quad\text{when $i\ne j$.}
\]
Thus the decomposition of $\gam_\fp k$ yields the first of the following isomorphism of rings:
\[
\End^*_{\cat T_E}(\gam_{\fp} k) \longiso  \bigoplus_{i=1}^n \End^*_{\cat T_E}(\gam_{\fq_i} k) \longiso  \bigoplus_{i=1}^n  H^*(E,k)^{^\wedge}_{\fq_i}\,.
\]
The second one is by \eqref{eq:benson-con}. It remains to observe that the ring on the right is naturally isomorphic to $H^*(E,k)\phat$. 

Finally, the decomposition above also allows one to decompose both sides of the map (2) in diagram \eqref{eq:square} into $\fq_i$-components and appeal to Theorem~\ref{th:spec-E}. This justifies the last part of the statement.
\end{proof}

Consider the category $\cat A(G)$ introduced above. We write $\cat A(G)_\fp$ for the subcategory consisting of elementary abelian $p$-subgroups on which $\fp$ is supported. The assignments
\[
E\mapsto \bSpec (\gam_{\fp}{\cat T}_E) \quad\text{and}\quad 
E\mapsto \Spec H^*(E,k)\phat
\]
define functors from $\cat A(G)_\fp$ to topological spaces, and the maps
\[
\bSpec (\gam_{\fp}{\cat T}_E) \longiso \Spec H^*(E,k)\phat
\]
define a natural transformation; the isomorphism is by Lemma~\ref{le:map2}. Moreover the vertical maps in \eqref{eq:square} are compatible with these maps. Summing, up we get the commutative square:
\begin{equation}
\label{eq:colimsquare-partial}
    \begin{tikzcd}
\underset{E\in \cat A(G)_\fp}{\colim} \bSpec \duals{(\gam_{\fp}{\cat T}_E)} \ar[d, "(1)" swap] \ar[r,"\sim" swap, "(2)"] 
    & \underset{E\in \cat A(G)_\fp}{\colim} \Spec H^*(E,k)\phat \ar[d, "(3)"] \\
 \bSpec \duals{(\gam_{\fp}{\cat T}_G)} \ar[r,twoheadrightarrow,"(4)" swap] &  \Spec H^*(G,k)\phat \,.
\end{tikzcd}
\end{equation}
This is exactly the diagram \eqref{eq:colimsquare}, except that we have restricted the colimit to $\cat A(G)_\fp$, and we are yet to justify that (1) is onto and  (3) is a homeomorphism.

\begin{lemma}
\label{le:map1}
The map \emph{(1)} in diagram \eqref{eq:colimsquare-partial} is surjective. 
\end{lemma}

\begin{proof}
It suffices to prove that the natural map 
\[
\coprod_{E\in \cat A(G)_\fp} \bSpec \duals{(\gam_{\fp}{\cat T}_E)} \longrightarrow \bSpec \duals{(\gam_{\fp}{\cat T}_G)}
\]
is surjective.  This map is induced by the functor
\begin{equation}
\label{eq:rho}    
\rho\colon \gam_{\fp}{\cat T}_G\longrightarrow \prod_{E\in\cat A(G)_\fp}\gam_{\fp}{\cat T}_E \cong \prod_{E\in\cat A(G)} \gam_{\fp} {\cat T}_E
\end{equation}
induced by the restriction functors $\rho_E$. The isomorphism holds because $\gam_\fp{\cat T}_E\equiv \{0\}$ if $\fp$ is not supported on $E$. By \cite[Theorem 1.3]{Balmer:2018a}, it suffices to verify that the functor $\rho$  is nil faithful, meaning that if a map $f\colon M\to N$ in $\gam_{\fp}\cat T_G$ is such that $\rho(f)=0$, then $f$ is tensor nilpotent, that is to say, $f^{\otimes n}=0$ for some $n\gg 0$. 

As noted before, $\gam_{\fp}{\cat T}_G$ is a full subcategory of $\cat T_G$, and this inclusion is compatible with tensor products; ditto for the $\gam_{\fp}{\cat T}_E$. Moreover the functor $\rho$ is induced by the restriction functor,  that we also denote $\rho$, on the whole stable module category:
\[
\rho\colon {\cat T}_G\longrightarrow \prod_{E\in\cat A(G)} {\cat T}_E\,.
\]
Thus it suffices to verify that this functor is nil faithful. 

Let $f\colon M\to N$ be a morphism in $\cat T_G$ such that $\rho_E(f)=0$ for $E\in \cat A(G)$.  A standard argument allows one to assume $M=k$. This goes as follows: $f$ corresponds to a map $f'\colon k\to \Hom_k(M,N)$ and $\rho_E(f)=0$ if and only if $\rho_E(f')=0$, since
\[
\rho_E k = k\quad\text{and}\quad \rho_E \Hom_k(M,N) = \Hom_k (\rho_EM,\rho_EN)\,.
\]
Also, if $f'$ is tensor nilpotent, then so is $f$, because the composition of maps
\[
    k \xrightarrow{(f')^{\otimes n}} \Hom_k(M,N)^{\otimes n}  \longrightarrow \Hom_k(M^{\otimes n},N^{\otimes n})
 \]
 assigns $1$ to $f^{\otimes n}$. Thus we can suppose that $M=k$.  Since $k$ is compact in $\cat T_G$, to verify that $f$ is tensor nilpotent it suffices to verify that $\rho$  is faithful on objects, that is to say, for $X$ in $\cat T_G$, if $\rho_EX=0$ for each $E\in \cat A(G)$, then $X=0$; see, for instance, \cite[Lemma~6.2]{Benson/Iyengar/Krause/Pevtsova:2021b}. It remains to recall that Chouinard~\cite[Theorem~1]{Chouinard:1976a} proved the faithfulness of $\rho$.
\end{proof}

\begin{lemma}
\label{le:map3}
The map \emph{(3)} in diagram \eqref{eq:colimsquare-partial} is a homeomorphism. 
\end{lemma}

\begin{proof}
By Quillen's theorem~\cite[Theorem~6.2]{Quillen:1971b} we have a uniform $F$-isomorphism 
 \[
 H^*(G,k) \longrightarrow  \lim_{E\in \cat A(G)} H^*(E,k). 
 \] 
 By this one means that the kernel of the map is nilpotent and that there exists an integer $n$ such that for any $\xi \in H^*(E,k)$, the element $\xi^{p^n}$ is in the image of the map. Localising at $\fp$ and then completing at $\fp$, and keeping in mind that the limit is over a finite diagram, yields an $F$-isomorphism 
 \[
 H^*(G,k)\phat \longrightarrow \lim_{E\in \cat A(G)_\fp} H^*(E,k)\phat.
 \]  
 Here we have used the fact that the map $H^*(G,k)\to H^*(G,k)\phat$ is flat, and that for any finitely generated $H^*(G,k)$-module $M$, there is a natural isomorphism
 \[
 H^*(G,k)\phat\otimes_{H^*(G,k)}M\longiso M\phat\,.
 \]
 The limit is now over the subcategory $\cat A(G)_\fp$ of $\cat A(G)$, for $H^*(E,k)_\fp=0$ when $E$ is not in this subcategory. Applying $\Spec(-)$ yields the isomorphism on the right:
 \[
\underset{E\in \cat A(G)_\fp}{\colim} \Spec H^*(E,k)\phat
    \longiso\Spec \lim_{E\in \cat A(G)_\fp} H^*(E,k)\phat
    \longiso\Spec H^*(G,k)\phat\,.
 \]
 The isomorphism on the left holds because $\cat A(G)_\fp$ is a finite category, and the rings $H^*(E,k)\phat$ are all finite over a fixed ring, namely $H^*(G,k)\phat$; see \cite[Lemma~8.11]{Quillen:1971c}. This is the desired statement.
\end{proof}

At this point we have completed the proof of Theorem~\ref{th:spec-G}.

\subsection*{Quillen stratification for local dualisable modules}
Fix $\fp$ in $\Proj H^*(G,k)$. In the course of the proof of Lemma~\ref{le:map3} we established a version of the Quillen stratification theorem for the spectrum of $\End^*_{\cat T_G}(\gam_\fp k)$,  the cohomology of the unit in the category of local dualisable modules. This is a formal consequence of the classical Quillen stratification theorem. There is an analogue of this result for arbitrary dualisable modules, described below; this requires more work.

Let $M$ be a dualisable module in $\gam_{\fp}\cat T_G$. For any subgroup $H\leq G$ write $I_H(M)$ for the kernel of the natural map
\[
R_H \coloneqq H^*(H,k)\phat \longrightarrow \End^*_{\cat T_H}(M{\downarrow_H})\,.
\]
Set $\supp_{R_H}(M)\coloneqq V(I_H(M))$; this is the cohomological support of $M{\downarrow_H}$, which is a dualisable object of $\gam_\fp\cat T_H$. Since the map above is compatible with restrictions to subgroups the $F$-isomorphism
\[
 R_G \longrightarrow \lim_{E\in \cat A(G)_\fp} R_E\,.
 \]
 established above, yields a map of rings
 \begin{equation}
 \label{eq:stratification}
 R_G/I_G(M) \longrightarrow \lim_{E\in \cat A(G)_\fp} R_E/I_E(M)    
 \end{equation}
 with the property that a $p$-power of any element of the target is in the image of the map; it also has a nilpotent kernel because the map $\rho$ in \eqref{eq:rho} is nil faithful. Thus we conclude that the map above is an $F$-isomorphism.
 Passing to spectra, yields the following result.

 \begin{theorem}
 \pushQED{\qed}
 \label{th:stratification}
 The map \eqref{eq:stratification} induces a homeomorphism of supports:    
 \[
 \colim_{E\in \cat A(G)_\fp} \supp_{R_E}(M{\downarrow_E}) \longiso \supp_{R_G}(M)\,. \qedhere
 \]
 \end{theorem}
 
 This is an analogue, for local dualisable modules, of the Quillen stratification for the cohomology of finite dimensional modules proved by Avrunin and Scott~\cite[Theorem~3.3]{Avrunin/Scott:1982a}.
 
\begin{ack}
We thank the American Institute of Mathematics for its hospitality and the support through the SQuaRE program. Our work was also supported by the National Science Foundation under Grant No.\ DMS-1928930 and by the Alfred P.\ Sloan Foundation under grant G-2021-16778, while three of the authors were in residence at the Simons Laufer Mathematical Sciences Institute in Berkeley, California, during the Spring 2024 semester. HK and JP thank the University of Utah for its hospitality during their visit in March of 2025.

SBI was partly supported by NSF grant DMS-2001368, HK was partly supported by the Deutsche Forschungsgemeinschaft (SFB-TRR 358/1 2023 - 491392403), and JP was partly supported by NSF grants DMS-1901854, 2200832 and the  Brian and Tiffinie Pang faculty fellowship.
\end{ack}

\bibliographystyle{amsplain}

\begin{thebibliography}{10}

\bibitem{Avramov/Buchweitz/Iyengar/Miller:2010a}
L.~L. Avramov, R.-O. Buchweitz, S.~B. Iyengar, and C.~Miller, \emph{{Homology
  of perfect complexes}}, Adv.\ in Math. \textbf{223} (2010), 1731--1781.

\bibitem{Avrunin/Scott:1982a}
G.~S. Avrunin and L.~L. Scott, \emph{{Quillen stratification for modules}},
  Invent.\ Math. \textbf{66} (1982), 277--286.

\bibitem{Balmer:2005a}
P.~Balmer, \emph{{The spectrum of prime ideals in tensor triangulated
  categories}}, J.~Reine \& Angew.\ Math. \textbf{588} (2005), 149--168.

\bibitem{Balmer:2010b}
\bysame, \emph{{Spectra, spectra, spectra---tensor triangular spectra versus
  Zariski spectra of endomorphism rings}}, Algebr.\ Geom.\ Topol. \textbf{10}
  (2010), no.~3, 1521--1563.

\bibitem{Balmer:2018a}
\bysame, \emph{{On the surjectivity of the map of spectra associated to a
  tensor-triangulated functor}}, Bull.\ London Math.\ Soc. \textbf{50} (2018),
  no.~3, 487--495.

\bibitem{Benson:2008a}
D.~J. Benson, \emph{{Idempotent $kG$-modules with injective cohomology}},
  J.~Pure \& Applied Algebra \textbf{212} (2008), 1744--1746.

\bibitem{Benson/Carlson/Rickard:1997a}
D.~J. Benson, J.~F. Carlson, and J.~Rickard, \emph{{Thick subcategories of the
  stable module category}}, Fundamenta Mathematicae \textbf{153} (1997),
  59--80.

\bibitem{Benson/Greenlees:2008a}
D.~J. Benson and J.~P.~C. Greenlees, \emph{{Localization and duality in
  topology and modular representation theory}}, J.~Pure \& Applied Algebra
  \textbf{212} (2008), 1716--1743.

\bibitem{Benson/Iyengar/Krause:2008a}
D.~J. Benson, S.~B. Iyengar, and H.~Krause, \emph{{Local cohomology and support
  for triangulated categories}}, Ann.\ Scient.\ \'Ec.\ Norm.\ Sup.\ (4)
  \textbf{41} (2008), 575--621.

\bibitem{Benson/Iyengar/Krause:2011b}
\bysame, \emph{{\null Stratifying modular representations of finite groups}},
  Ann.\ of Math. \textbf{174} (2011), 1643--1684.

\bibitem{Benson/Iyengar/Krause:2012b}
\bysame, \emph{{Colocalising subcategories and cosupport}}, J.~Reine \& Angew.\
  Math. \textbf{673} (2012), 161--207.

\bibitem{Benson/Iyengar/Krause:2015a}
\bysame, \emph{{A local--global principle for small triangulated categories}},
  Math.\ Proc.\ Camb.\ Phil.\ Soc. \textbf{158} (2015), no.~3, 451--476.

\bibitem{Benson/Iyengar/Krause/Pevtsova:2019a}
D.~J. Benson, S.~B. Iyengar, H.~Krause, and J.~Pevtsova, \emph{{Local duality
  for representations of finite group schemes}}, Compositio Mathematica
  \textbf{155} (2019), 424--453.

\bibitem{Benson/Iyengar/Krause/Pevtsova:2021b}
\bysame, \emph{{Detecting nilpotence and projectivity over finite supergroup
  schemes}}, Selecta Math. \textbf{27} (2021), 25pp.

\bibitem{Benson/Iyengar/Krause/Pevtsova:2024a}
\bysame, \emph{{Fibrewise stratification of group representations}}, Annals of
  Representation Theory \textbf{1} (2024), no.~1, 97--124.

\bibitem{Benson/Iyengar/Krause/Pevtsova:2024c}
\bysame, \emph{{Local dualisable modular representations and local
  regularity}}, \href{https://arxiv.org/abs/2404.14672}{arXiv:2404.14672},
  2024.

\bibitem{Benson/Iyengar/Krause/Pevtsova:2024b}
\bysame, \emph{{Local dualisable objects in local algebra}}, {Triangulated
  Categories in Representation Theory and Beyond} (P.~A. Bergh, S.~Oppermann,
  and \O. Solberg, eds.), Abel Symposia, vol.~17, Springer-Verlag, Ber\-lin/New
  York, 2024.

\bibitem{Carlson/Iyengar:2015a}
J.~F. Carlson and S.~B. Iyengar, \emph{{Thick subcategories of the bounded
  derived category of a finite group}}, Trans.\ Amer.\ Math.\ Soc. \textbf{367}
  (2015), no.~4, 2703--2717.

\bibitem{Chouinard:1976a}
L.~Chouinard, \emph{{Projectivity and relative projectivity over group rings}},
  J.~Pure \& Applied Algebra \textbf{7} (1976), 278--302.

\bibitem{Hochster:1969a}
M.~Hochster, \emph{{Prime ideal structure in commutative rings}}, Trans.\
  Amer.\ Math.\ Soc. \textbf{142} (1969), 43--60.

\bibitem{Hovey/Palmieri/Strickland:1997a}
M.~Hovey, J.~H. Palmieri, and N.~P. Strickland, \emph{{Axiomatic stable
  homotopy theory}}, vol. 128, Mem.\ Amer.\ Math.\ Soc., no. 610, American
  Math.\ Society, 1997.

\bibitem{Johnstone:1982a}
P.~Johnstone, \emph{{Stone spaces}}, Cambridge Studies in Advanced Mathematics,
  vol.~3, Cambridge University Press, 1982.

\bibitem{Kock/Pitsch:2017a}
J.~Kock and W.~Pitsch, \emph{{Hochster duality in derived categories and
  point-free reconstruction of schemes}}, Trans.\ Amer.\ Math.\ Soc.
  \textbf{369} (2017), no.~1, 223--261.

\bibitem{Krause:2023a}
H.~Krause, \emph{{Central support for triangulated categories}}, Int.\ Math.\
  Res.\ Not. (2023), no.~22, 19773--19800.

\bibitem{Quillen:1971a}
D.~G. Quillen, \emph{{A cohomological criterion for $p$-nilpotence}}, J.~Pure
  \& Applied Algebra \textbf{1} (1971), 361--372.

\bibitem{Quillen:1971b}
\bysame, \emph{{The spectrum of an equivariant cohomology ring, I}}, Ann.\ of
  Math. \textbf{94} (1971), 549--572.

\bibitem{Quillen:1971c}
\bysame, \emph{{The spectrum of an equivariant cohomology ring, II}}, Ann.\ of
  Math. \textbf{94} (1971), 573--602.

\bibitem{Stevenson:2013a}
G.~Stevenson, \emph{{Support theory via actions of tensor triangulated
  categories}}, J.~Reine \& Angew.\ Math. \textbf{681} (2013), 219--254.

\end{thebibliography}
\newcommand{\noopsort}[1]{}
\providecommand{\bysame}{\leavevmode\hbox to3em{\hrulefill}\thinspace}
\providecommand{\MR}{\relax\ifhmode\unskip\space\fi MR }

\providecommand{\MRhref}[2]{%
  \href{http://www.ams.org/mathscinet-getitem?mr=#1}{#2}
}
\providecommand{\href}[2]{#2}

\end{document}